\documentclass{conm-p-l}

\usepackage{amsthm,amsfonts, cite}
\usepackage{amssymb,graphicx,color, mathdots}
\usepackage[all]{xy}
\usepackage[sc]{mathpazo}
\usepackage{dsfont}
\usepackage{eucal}
\usepackage{enumerate}
\usepackage{hyperref}
\usepackage{csquotes}
\numberwithin{equation}{section}

\theoremstyle{plain}
\newtheorem{teo}{Theorem}[section]
\newtheorem{cor}[teo]{Corollary}
\newtheorem{lem}[teo]{Lemma}

\newtheorem{teoA}{Theorem}

\newtheorem{corA}[teoA]{Corollary}

\theoremstyle{remark}
\newtheorem{obs}[teo]{Remark}

\newcommand{\R}{\ensuremath{{\mathbb{R}}}}
\newcommand{\C}{\ensuremath{{\mathbb{C}}}}

\newcommand{\Z}{\ensuremath{{\mathds{Z}}}}
\newcommand{\g}{\ensuremath{\mathtt{g}}}
\newcommand{\hg}{\ensuremath{\widehat{\mathtt{g}}}}
\newcommand{\tg}{\ensuremath{\widetilde{\mathtt{g}}}}
\renewcommand\vec[1]{\boldsymbol{#1}}
\newcommand{\dispdot}[2][.2ex]{\dot{\raisebox{0pt}[\dimexpr\height+#1][\depth]{$#2$}}} 

\begin{document}

\title[The metric structure of compact rank-one ECS manifolds]{The metric structure of compact rank-one ECS manifolds}

\author[Andrzej Derdzinski]{Andrzej Derdzinski}
\author[Ivo Terek]{Ivo Terek}
\address{Department of Mathematics, The Ohio State University, Columbus, OH 43210, USA}
\email{andrzej@math.ohio-state.edu}
\email{terekcouto.1@osu.edu}



\begin{abstract}
Pseudo-Riemannian manifolds with nonzero parallel Weyl tensor \linebreak[4]which are not locally symmetric are known as ECS manifolds. Every ECS manifold carries a distinguished null parallel distribution $\mathcal{D}$, the rank $d\in \{1,2\}$ of which is referred to as the rank of the manifold itself. Under a natural genericity assumption on the Weyl tensor, we fully describe the universal coverings of compact rank-one ECS manifolds. We then show that any generic compact rank-one ECS manifold must be \emph{translational}, in the sense that the holonomy group of the natural flat connection induced on $\mathcal{D}$ is either trivial or isomorphic to $\Z_2$. We also prove that all four-dimensional rank-one ECS manifolds are noncompact, this time without assuming genericity, as it is always the case in dimension four.
\end{abstract}
\maketitle

\section*{Introduction and main results}

A pseudo-Riemannian manifold of dimension $n\geq 4$ whose Weyl tensor is parallel is referred to as \emph{conformally symmetric} \cite{Chaki-Gupta}, and it is called \emph{essentially conformally symmetric} (briefly, \emph{ECS}) \cite{Roter74} if, in addition, it is neither conformally flat nor locally symmetric. It was shown by Roter that ECS manifolds exist in every dimension \cite[Corollary 3]{Roter74} and that they all have indefinite metric signatures \cite[Theorem 2]{TensorNS77}. The local structure of ECS manifolds is fully known \cite{local_structure}.

Conformal symmetry of $(M,\g)$ is one of the \emph{natural linear conditions} imposed on the covariant derivatives of the ${\rm SO}(p,q)$-irreducible components of its curvature tensor, in the sense of Besse \cite[Chapter 16]{besse}. The interest in this subject is reflected in more work by other authors: Cahen and Kerbrat \cite[Section 2]{Cahen-Kerbrat}, Hotlo\'{s} \cite{hotlos},  Mantica and Suh \cite[Section 3]{mantica-suh}, Schliebner \cite{schliebner}, and Deszcz et al. in \cite[Sect.\ 4]{deszcz-glogowska-hotlos-zafindratafa}, \cite[Theorem 6.1]{deszcz-glogowska-hotlos-torgasev-zafindratafa}. The techniques used in the study of ECS manifolds are themselves also of interest, appearing in \cite[Example 2.2]{deszcz-hotlos}, \cite{suh-kwon-yang}, \cite[Theorem 3]{alekseevsky-galaev}, \cite{calvino-louzao-garcia-rio-seoane-bascoy-vazquez-lorenzo}, \cite[Theorem 3.9]{calvino-louzao-garcia-rio-vazquez-abal-vazquez-lorenzo}, \cite[Lemma 3]{leistner-schliebner}, \cite{mantica-molinari}, \cite[proofs of Theorems 1.1 and 4.5]{tran} and \cite{terek}.

As shown by Olszak \cite{Olszak1}, every ECS manifold $(M,\g)$ carries a distinguished null parallel distribution $\mathcal{D}$, whose sections are the vector fields corresponding under $\g$ to $1$-forms $\xi$ with $\xi\wedge [{\rm W}(v',v'',\cdot,\cdot)]=0$ for all vector fields $v',v''$. The rank of $\mathcal{D}$ -- always equal to $1$ or $2$ -- is referred to as the \emph{rank of $(M,\g)$} \cite{DT_2}. In the rank-one case, the focus of this paper, we also call the ECS manifold in question
\begin{equation}\label{eqn:dichotomy}
  \parbox{.8385\textwidth}{\emph{translational} or \emph{dilational}, depending on whether the holonomy group of \,the\, natural\, flat\, connection\, induced\, on\, $\,\mathcal{D}\,$\, is\, finite\, or\, infinite.}\tag{$\ast$}
\end{equation}

In \cite{DT_1} and \cite{AGAG10}, compact rank-one ECS manifolds of every dimension $n\geq 5$ were constructed as suitable quotients of what we call \emph{model ECS manifolds} (see Section \ref{sec:model}). All such compact examples are geodesically complete and translational, but none of them is locally homogeneous. On the other hand, we show in \cite{DT_3} that, under a natural genericity assumption on the Weyl tensor (see Section \ref{sec:generic} and the end of Section \ref{sec:comp_Dperp}), quotients of dilational model ECS manifolds cannot be compact unless they are locally homogeneous. We have recently found \cite{clh} compact dilational examples in all odd dimensions $n\geq 5$, including locally homogeneous ones. They are all nongeneric and incomplete.

It is still not known whether a compact ECS manifold can be four-dimensional, or have rank two.

This paper provides partial answers to the above questions. We start with structure theorems, calling a rank-one ECS manifold \emph{$\mathcal{D}^\perp\!$-complete} if all the leaves of $\mathcal{D}^\perp\!$ are complete relative to the induced connections (this definition makes sense for any foliation with totally geodesic leaves, such as $\mathcal{D}^\perp\!$ itself), and \emph{maximally com\-ple\-te} if every non-com\-ple\-te maximal geodesic in its universal covering intersects all leaves of $\mathcal{D}^\perp\!$. Note that maximal completeness implies $\mathcal{D}^\perp\!$-completeness, while (due to \eqref{eqn:t-levels} below) it follows from completeness.

\begin{teoA}\label{teoA:generic_mc}
Any compact $\mathcal{D}^\perp\!$-complete rank-one ECS manifold is necessarily maximally complete.
\end{teoA}

\begin{teoA}\label{teo:generic_Dperp}
  Every generic compact rank-one ECS manifold is both maximally complete and $\mathcal{D}^\perp\!$-complete.
\end{teoA}

\begin{teoA}\label{teoA:uc_model}
 Any simply connected and maximally complete rank-one ECS manifold is isometric to a model ECS manifold.
\end{teoA}

The next result is a trivial consequence of Theorems \ref{teo:generic_Dperp} and \ref{teoA:uc_model}.

\begin{corA}\label{teoA:uc_cpct_model}
  The universal covering of any generic compact rank-one ECS manifold is isometric to a model ECS manifold.
\end{corA}

For the Lorentzian signature, Schliebner \cite{schliebner} proved this last conclusion without assuming genericity. From Corollary \ref{teoA:uc_cpct_model}, we obtain the following strengthened version of \cite[Theorem C]{DT_3}, which refers to the dichotomy \eqref{eqn:dichotomy}:

\begin{teoA}\label{teoA:gen_cpt_translational}
  Every generic compact rank-one ECS manifold is translational, as well as geodesically complete, and it cannot be locally homogeneous.
\end{teoA}

Let us point out that Theorem \ref{teoA:gen_cpt_translational} does not replace \cite[Theorem C]{DT_3}, but rather relies on it, since the latter is needed to prove the former. As we point out in Section \ref{sec:4dim_goodbye}, Corollary \ref{teoA:uc_cpct_model} combined with \cite[Theorem 8.1]{AGAG10} trivially leads to:


\begin{corA}\label{cor:dim4_goodbye}
Four-dimensional rank-one ECS manifolds are noncompact.  
\end{corA}

In other words, if four-dimensional compact ECS manifolds do exist, they must necessarily be of rank two.

\section*{How the paper is organized}

Unless stated otherwise, all manifolds, bundles, connections, mappings, and tensor fields are assumed to be smooth. The text is divided into two parts.

\medskip

\noindent \textbf{Part I.} After Sections \ref{sec:prelim-conn}--\ref{sec:prelim-spectra} dealing with preliminaries, in Section \ref{sec:generic} we elaborate on the meaning of genericity. Sections \ref{sec:comp_Dperp} and \ref{sec:model} lay the groundwork for proving Theorems \ref{teoA:generic_mc} and \ref{teoA:uc_model}, summarizing what is already known about the structure of the universal coverings of compact rank-one ECS manifolds, and describing our model ECS manifolds. In Section \ref{sec:main_proofs} we prove Theorems \ref{teoA:generic_mc}--\ref{teoA:uc_model}, adapting to our situation the proofs of some weaker results from \cite{JGP08} (namely, Lemma 7.3 and Theorem 7.1 therein). Further details are provided in Appendices \ref{app:technical} and \ref{app:model}.

\medskip

\noindent \textbf{Part II.} Section \ref{sec:TCP} introduces the \emph{transitive-com\-mu\-ta\-ti\-on property} for a group-sub\-group pair, crucial for understanding the structure of the isometry group of a locally homogeneous model ECS manifold. The focus on the locally homogeneous case is justified by \cite[Theorem C]{DT_3}, where we prove that a generic dilational compact rank-one ECS manifold must be locally homogeneous. Section \ref{sec:LH_details} presents what we call \emph{standard homogeneous rank-one ECS model manifolds}, used in Section \ref{sec:proof_E} to prove Theorem \ref{teoA:gen_cpt_translational}.

\section{Completeness of connections}\label{sec:prelim-conn} In this section, we fix a manifold $L$ equipped with a connection $\nabla$.

\begin{lem}\label{lem:IVP_affine}
 If $\vec{X}$ and $\vec{Z}$ are vector fields along a curve defined on an open interval $J\subseteq\R$ of the
variable $s$ containing $0$, and
  \begin{equation}\label{eqn:IVP_affine}
   \nabla_{\!s}\vec{Z} = \nabla_{\!s}\nabla_{\!s} \vec{X} = 0,\quad \vec{X}(0) = \vec{Z}(0),\quad [\nabla_{\!s}\vec{X}](0) = -\vec{Z}(0),
  \end{equation}
then $\vec{X}(s) = (1-s)\vec{Z}(s)$ for all $s\in J$.
\end{lem}

In fact, $s\mapsto \vec{X}_s= (1-s)\vec{Z}_s$ satisfies \eqref{eqn:IVP_affine}.

The next Lemma generalizes \cite[Lemma 1.4]{JGP08} and is used in Section \ref{sec:main_proofs} to prove Theorem \ref{teo:generic_Dperp}. In its proof (and later in Appendix \ref{app:technical}) we adopt the notational
convention of \cite[end of Sect. 1]{JGP08}: given a \emph{variation of curves}
in a manifold $M$, that is, a ${\rm C}^\infty$
mapping $(t,s)\mapsto x(s,t)$ from an open set in $\R^2$ into $M$, and a connection on $M$, we denote by $x_s,x_t$ (or, $x_{ss},x_{st},x_{tss}$, etc.) its
partial (or, partial covariant) derivatives of orders $1,2,3$ etc., all
of which are vector fields along the variation, meaning, as usual, sections of
the corresponding pull\-back of $TM$. As the connections involved
are flat and tor\-sion-free,
\begin{equation}\label{eqn:derivatives_symmetric}
  \parbox{.74\textwidth}{all such derivatives depend symmetrically on the subscripts.}
\end{equation}

\begin{lem}\label{lem:abstract_completeness}
  Let $\mathcal{P}$ be a distribution on $L$, and assume that:
  \begin{enumerate}[\normalfont(a)]
  \item $\nabla$ is flat and torsionfree,
  \item $\mathcal{P}$ is trivialized by a vector space $\mathcal{X}$ of complete parallel vector fields,
  \item there is a vector space $\mathcal{Y}$ of complete vector fields on $L$ which is isomorphically mapped via the quotient projection onto a vector space of sections trivializing the quotient bundle $TL/\mathcal{P}$ over $L$, and parallel relative to the connection induced on $TL/\mathcal{P}$.
  \end{enumerate}
   Then $\nabla$ is complete.
\end{lem}

\begin{proof}
  The evaluation $\mathcal{X}\to \mathcal{P}_z$ at each $z\in L$ is an isomorphism, and so every $\nabla$-geodesic of $L$ starting tangent to $\mathcal{P}$ is complete, such a geodesic being an integral curve of a complete parallel vector field on $L$.

Given $z\in L$ and ${\bf v} \in \mathcal{Y}$, let $y\colon \R\to L$ be the integral curve of ${\bf v}$ with $y(0)=z$. As $\nabla_{\!{\bf v}}{\bf v}$ is always tangent to $\mathcal{P}$, we may choose $\zeta,\eta\colon \R \to \mathcal{X}$ with
\begin{equation}\label{eqn:zeta_and_eta}
  [\zeta(t)]_{y(t)} = [\nabla_{\!{\bf v}}{\bf v}]_{y(t)}=\ddot{y}(t),\quad \ddot{\eta} = -\zeta.
\end{equation}
Each $\eta(t)$ is complete, and the variation $\R^2\ni (t,s) \mapsto x(t,s) = {\rm e}^{s\eta(t)}y(t)$ has
  \begin{equation}\label{eqn:conditions_var}
   x(t,0) = y(t),\quad x_s(t,0) = \left[\eta(t)\right]_{y(t)},\quad x_{ss}(t,s)= 0.
  \end{equation}
Hence $x_{tt}(t,s)=(1-s)[\zeta(t)]_{x(t,s)}$ for all $(t,s)\in\R^2$, as
one sees applying Lemma \ref{lem:IVP_affine} to $\vec{Z}(s)=[\zeta(t)]_{x(t,s)}$ and $\vec{X}(s)=x_{tt}(t,s)$, with fixed $t$, the equalities \eqref{eqn:IVP_affine} being immediate from \eqref{eqn:derivatives_symmetric}--\eqref{eqn:conditions_var}. (Note that $\nabla[\zeta(t)]=0$ and $x_{ttss}=x_{sstt}=0$.) In particular, $t\mapsto x(t,1) = {\rm e}^{\eta(t)}y(t)$ is a complete geodesic whose initial velocity is, when $\eta(0) = 0$, equal to ${\bf v}_z + \dot{\eta}(0)$. However, every vector in $T_zM$ is of this form for suitable ${\bf v}$ and $\eta$, as the values ${\bf v}_z$ realize all values at $z$ in the complementary subbundle to $\mathcal{P}$ spanned by $\mathcal{Y}$, while the values $\dot{\eta}(0)$ realize all elements of $\mathcal{P}_z$.
\end{proof}

\begin{lem}\label{lem:complete_surjective}
  If $\nabla$ is complete and $f\colon L\to\R$ has \hbox{$\nabla {\rm d}f = 0$}, then $f$ is surjective.
\end{lem}

Namely, along a maximal geodesic, $f$ is an affine function of its parameter.

\section{Properly discontinuous $\R^k$-subactions}\label{sec:prelim-pd}

Three well-known facts are phrased here as a remark for easy reference.

\begin{obs}\label{rem:comp_fibrations}
First, the composition of two fibrations (including covering projections) is clearly a fibration. Secondly, if a Lie group $G$ acts on a manifold $\widehat{M}$ with a subgroup $\Gamma$ acting on $\widehat{M}$ freely and properly discontinuously, then $\Gamma$ is a discrete subset of $G$. Finally, whenever a compact manifold is contractible, it consists of a single point. (Otherwise, being simply connected, and hence orientable, it would have a nontrivial top cohomology group.)
\end{obs}

\begin{lem}\label{lem:free_principal}
  If $\R^k$ acts freely on a contractible manifold $\widehat{M}$ and a subgroup $\Gamma$ of  $\R^k$ acts on $\widehat{M}$ properly discontinuously with a compact quotient $\widehat{M}/\Gamma$, then $k = \dim \widehat{M}$, the action of $\R^k$ on $\widehat{M}$ is simply transitive, and $\Gamma$ is a lattice in $\R^k$. Consequently, $\widehat{M}$ and $\widehat{M}/\Gamma$ are, respectively, an affine $k$-space and a $k$-dimensional torus.
\end{lem}

\begin{proof}
As $\Gamma$ is a discrete subset of $\R^k$ (Remark \ref{rem:comp_fibrations}), it forms a lattice in the subspace $Y \subseteq \R^k$ which it spans and, due to commutativity, the action of $\R^k$ on $\widehat{M}$ descends to a free action of the torus $Y/\Gamma$ on $\widehat{M}/\Gamma$ which, according to \cite[Corollary 4.2.11, p. 213]{hamilton_gauge}, turns $\widehat{M}/\Gamma$ into the total space of a principal torus bundle over some compact base $B$. By Remark \ref{rem:comp_fibrations}, the composition \hbox{$\widehat{M} \to \widehat{M}/\Gamma \to B$} is a bundle projection with the fibre $Y$, and its homotopy long exact sequence \cite[Theorem 4.49, p. 376]{hatcher} implies that $B$ has trivial homotopy groups, being therefore contractible \cite[Lemma 2.1]{luft}. Due to Remark \ref{rem:comp_fibrations}, $B$ consists of a single point and the resulting relations $\dim \widehat{M} = \dim Y \le k \le \dim \widehat{M}$, the last one immediate since the action of $\R^k$ is free, yield our assertion.  
\end{proof}

\section{Spectra of endomorphisms}\label{sec:prelim-spectra}

Given an endomorphism $B$ of a $k$-di\-men\-si\-o\-nal vector space $\mathcal{X}$, by the \emph{spectrum} of $B$ we mean the unordered system $\beta(1),\ldots, \beta(k)$ formed by the complex characteristic roots of $B$ listed with their multiplicities. If $B = [{\rm d}\sigma_q/{\rm d}q]_{q=1}$ is the infinitesimal generator of a Lie-group homomorphism \hbox{$(0,\infty)\ni q\mapsto \sigma_q \in {\rm GL}(\mathcal{X})$} and the spectrum of each $\sigma_q$ is $q^{\alpha(1)},\ldots,q^{\alpha(k)}$, with $q^{\alpha(j)} = q^{{\rm Re}\,\alpha(j)} {\rm e}^{{\rm i}(\log q){\rm Im}\,\alpha(j)}$, where $\alpha(j)\in \C$ \emph{do not depend on $q$}, for $j=1,\ldots,k$, then
\begin{equation}\label{eqn:spec_B_general}
  \parbox{.42\textwidth}{$B\,$ has the spectrum $\alpha(1),\ldots,\alpha(k)$.}
\end{equation}
In fact, the complex-linear extension of $B$ to $\mathcal{X}^{\C}$ has, in some basis, an upper triangular matrix with the diagonal entries $\beta(1),\ldots,\beta(k)$ forming the spectrum of $B$. Thus $\sigma_q = \exp[(\log q)B]$ has the spectrum $q^{\beta(j)}$, $j=1,\ldots, k$, and, up to a rearrangement, $q^{\beta(j)} = q^{\alpha(j)}$. Hence $[\beta(j)-\alpha(j)]\log q \in 2\pi{\rm i}\Z$ and so $\beta(j)=\alpha(j)$.

\smallskip

It is a trivial fact from linear algebra that, whether $\mathcal{X}$ is finite-dimensional or not, every family $\mathcal{F}$ of eigen\-vectors of an endomorphism $\varPsi \in {\rm End}(\mathcal{X})$ corresponding to mutually distinct eigen\-values is linearly independent. As a consequence:

\begin{lem}\label{lem:eigenvectors}
 Given $\varPsi$ and $\mathcal{F}$ as above, let $(x_\alpha)_{\alpha\in A}$ be an \emph{indexed} family of vectors such that $x_\alpha\in\mathcal{F}$ whenever $\alpha\in A$. If
$A_0\subseteq A$ is a nonempty finite set with the property that $\sum_{\,\alpha\in A_0}x_\alpha\in\ker\varPsi$, then $x_\alpha\in\ker\varPsi$ for every $\alpha\in A_0$.
\end{lem}

\begin{proof}
  There are positive integers $n$ and $k_1,\ldots, k_n$, as well as $\alpha_1,\ldots,\alpha_n\in A$, such that $\sum_{\,\alpha \in A_0} x_\alpha = \sum_{i=1}^n k_ix_{\alpha_i}$, where $x_{\alpha_i}\neq x_{\alpha_j}$ whenever $i\neq j$. If $\lambda_i$ is the eigenvalue of $\varPsi$ associated with $x_{\alpha_i}$, it follows that $\sum_{i=1}^n k_i\lambda_ix_{\alpha_i} = 0$, whence $\lambda_i = 0$. Therefore $n=1$ and $\lambda_1=0$.
\end{proof}

  The assumption and conclusion of Lemma \ref{lem:eigenvectors} apply to
$\varPsi=q \,{\rm d}/{\rm d}q$ in the space of all com\-plex-val\-ued
${\rm C}^\infty$ functions of the variable $q\in(0,\infty)$ and the family
$\mathcal{F}$ formed by all the power functions \hbox{$(0,\infty)\ni q\mapsto q^{a+b{\rm i}}=q^a{\rm e}^{{\rm i}b\log q}$ with} $a,b\in \R$. The proof of Theorem \ref{teoA:gen_cpt_translational}, in Section \ref{sec:proof_E}, uses the following consequence:
\begin{equation}\label{eqn:sum_powers_cte}
  \parbox{.697\textwidth}{the sum of several terms of the form $q^{a+b{\rm i}}$ can be constant as a function of $q$ only if each term in the sum is $q^0=1$.}
\end{equation}

\section{Generic endomorphisms}\label{sec:generic}

Throughout this section, we let $(V,\langle\cdot,\cdot\rangle)$ be a pseudo-Euclidean vector space of dimension $m$, denote by $\mathcal{A}$ the space of all traceless self-adjoint endomorphisms of $V$, and say that $A\in \mathcal{A}$ is \emph{generic} if only finitely many linear isometries of $(V,\langle\cdot,\cdot\rangle)$ commute with $A$. For $m = 2$, unless $A=0$, there are at most four linear isometries commuting with $A$, cf. \cite[Remark 6.2]{JGP08}, and so
\begin{equation}\label{eqn:generic_m2}
  \parbox{.52\textwidth}{every $A \in \mathcal{A}\smallsetminus \{0\}$ is generic when $m=2$.}
\end{equation}
In all dimensions $m$, generic en\-do\-mor\-phi\-sms always exist. In fact, any $A\in \mathcal{A}$ with $m$ distinct eigenvalues is generic, since its $m$ eigenlines are mutually orthogonal and hence nondegenerate. Furthermore, \begin{equation}\label{eqn:generic_dense} \parbox{.85\textwidth}{the set of generic endomorphisms is an open and dense subset of $\mathcal{A}$.}\end{equation} Indeed, note that an endomorphism $A\in \mathcal{A}$ is generic if and only if its isotropy group $G_A$ under the action of ${\rm O}(V,\langle\cdot,\cdot\rangle)$ on $\mathcal{A}$ by conjugation is finite. However, finiteness of $G_A$ amounts to its being countable, since $G_A$, given by the polynomial equation $CAC^*=A$, is an algebraic variety and so it has finitely many connected components \cite[Theorem 3]{Whitney57}. Consequently, genericity of $A$ is equivalent to triviality of its isotropy algebra, which in turn means that the rank of $F(A)$ equals $r$, for $r = \dim {\rm O}(V,\langle\cdot,\cdot\rangle)$ and $F\colon \mathcal{A} \to {\rm Hom}(\mathfrak{so}(V,\langle\cdot,\cdot\rangle), \mathcal{A})$ given by $F(A)(B) = [A,B]$. We now fix some generic endomorphism $A_0\in \mathcal{A}$, some bases of $\mathfrak{so}(V,\langle\cdot,\cdot\rangle)$ and $\mathcal{A}$, and a nonzero $r\times r$ subdeterminant of the matrix representing $F(A_0)$ in these bases. By analyticity, such a subdeterminant is nonzero on an open and dense subset of $\mathcal{A}$, thus proving \eqref{eqn:generic_dense}. 

\smallskip

In \cite[Section 4]{DT_3} we show that a nilpotent endomorphism $A\in\mathcal{A}$ is generic if and only if $A^{m-1}\neq 0$, in which case there is a basis $(v_1,\ldots, v_m)$ of $V$ such that
\begin{equation}\label{eqn:fit_A_g}
  \parbox{.83\textwidth}{$Av_j = v_{j-1}$ and $\langle v_i, v_k\rangle = \varepsilon\hskip1pt\delta_{ij}$ for all $i,j\in \{1,\ldots, m\}$, where $\varepsilon =\pm 1$ is the semi-definiteness sign of $\langle A^{m-1}\cdot,\cdot\rangle$, $k=m+1-j$, and $v_0=0$. In addition, such a basis is also unique up to an overall sign change.}
\end{equation}
It follows \cite[Corollary 4.3]{DT_3} that for every $q\in (0,\infty)$,
\begin{equation}\label{eqn:two_isometries}
  \parbox{.72\textwidth}{there are only two linear isometries $C,-C$ of $(V,\langle\cdot,\cdot\rangle)$ with \hbox{$CAC^{-1} = q^2A$} and, in a basis satisfying \eqref{eqn:fit_A_g}, they are\hskip4.2pt given\hskip4.2pt by\hskip4.2pt \hbox{$Cv_j = \delta q^{m+1-2j}v_j$},\hskip4.2pt for\hskip4.2pt some\hskip4.2pt sign\hskip4.2pt factor\hskip4.2pt $\delta =\pm 1$.}
\end{equation}

\section{The universal coverings}\label{sec:comp_Dperp}

In this section we fix a rank-one ECS manifold $(M,\mathtt{g})$ of dimension $n\geq 4$ with arbitrary indefinite metric signature, and let $\Gamma$ be the fundamental group of $M$. Consider the universal covering projection $\pi\colon \widetilde{M} \to M$, and set $\tg = \pi^*\g$, so that $(\widetilde{M},\tg)$ is a simply connected rank-one ECS manifold on which $\Gamma$ acts freely and properly discontinuously by isometries, with quotient $M = \widetilde{M}/\Gamma$. We will also write $\widetilde{\mathcal{D}}$ for the Olszak distribution of $(\widetilde{M},\tg)$, defined in the Introduction.

As the Levi-Civita connection of $(\widetilde{M},\tg)$ induces a connection on $\widetilde{\mathcal{D}}$, and the latter is flat \cite[Lemma 2.2(f)]{local_structure}, simple connectivity of $\widetilde{M}$ allows to us to fix
\begin{equation}\label{eqn:npg_D}
  \parbox{.81\textwidth}{a null parallel vector field ${\bf w}$ spanning $\widetilde{\mathcal{D}}$, leading to a surjective function $t\colon \widetilde{M} \to I$ onto an open interval $I\subseteq\R$, with \hbox{${\rm d}t = \tg({\bf w},\cdot)$}.}
\end{equation}
In addition, as shown in \cite[end of Section 11]{DT_2},
\begin{equation}\label{eqn:t-levels}
  \parbox{.7\textwidth}{the leaves of $\widetilde{\mathcal{D}}^\perp$ coincide with the level sets of $t\colon\widetilde{M}\to I$.}
\end{equation}
By Lemma \ref{lem:complete_surjective}, $(M,\g)$ is incomplete when $I\neq \R$. Moreover, as \emph{the Olszak distribution is a local geometric invariant of the given ECS metric}, $t$ in \eqref{eqn:npg_D} is unique up to affine substitutions, and so for every \hbox{$\gamma \in {\rm Iso}(\widetilde{M},\tg)$} there is $(q,p)\in {\rm Aff}(\R)$ such that $t\circ \gamma = qt+p$, giving rise to two homomorphisms:
\begin{equation}\label{eqn:hom_qp_q}
  \parbox{.89\textwidth}{a)\,\, ${\rm Iso}(\widetilde{M},\tg)\ni \gamma \mapsto (q,p)\in {\rm Aff}(\R)$, \quad b)\,\, ${\rm Iso}(\widetilde{M},\tg)\ni \gamma \mapsto q \in \R \smallsetminus \{0\}$.}
\end{equation} 

The following principle will be repeatedly used:
\begin{equation}\label{eqn:passing_to}  \parbox{.825\textwidth}{replacing $\Gamma$ with a finite index subgroup $\Gamma_{\!0}$ amounts to replacing $M$ with the quotient $\widetilde{M}/\Gamma_{\!0}$, which is also compact (as the total space of a finite-sheeted covering of $M$) and has $\widetilde{M}$ as its universal covering.}  \end{equation}

Using \eqref{eqn:passing_to}, we from now assume that
\begin{equation}\label{eqn:Dperp_tr_orient}
  \parbox{.62\textwidth}{the image of $\Gamma$ under (\ref{eqn:hom_qp_q}-b) is contained in $(0,\infty)$.}
\end{equation}

In \cite[Section 11]{DT_2}, we show that, if \eqref{eqn:Dperp_tr_orient} holds and $M$ is compact,
\begin{equation}\label{eqn:psi-dt}
  \parbox{.667\textwidth}{\begin{enumerate}[a)] \item there exists a smooth positive function $\psi$ on $\widetilde{M}$ such that the $1$-form $\psi\,{\rm d}t$ is closed and $\Gamma$-invariant, and \item the vector field ${\bf w}$ in \eqref{eqn:npg_D} is complete. \end{enumerate}}
\end{equation}
With $q$ related to $\gamma$ as in (\ref{eqn:hom_qp_q}-b), these $\psi$ and ${\bf w}$ satisfy the conditions
\begin{equation}\label{eqn:psi-q}
  \parbox{.6\textwidth}{$\psi\circ \gamma = q^{-1}\psi$ and $\gamma_\ast{\bf w} = q^{-1}{\bf w}$,\, for every $\gamma \in \Gamma$,}
\end{equation}due to the relation $\gamma^*({\rm d}t) = q\,{\rm d}t$ and $\Gamma$-invariance of $\psi\,{\rm d}t$.

The Levi-Civita connection of $(\widetilde{M},\tg)$ induces one on the quotient bundle $\widetilde{\mathcal{D}}^\perp\hskip-2pt/\widetilde{\mathcal{D}}$, which is flat by \cite[Lemma 2.2(f)]{local_structure}. Thus, $\widetilde{M}$ being simply connected, the real vector space $V$ of parallel sections of $\widetilde{\mathcal{D}}^\perp\hskip-2pt/\widetilde{\mathcal{D}}$ has the full dimension $m = n-2$. The space $V$ also inherits a natural pseudo-Euclidean inner product $\langle\cdot,\cdot\rangle$ from $\tg$, and the Weyl tensor ${\rm W}$ of $(\widetilde{M},\tg)$ induces, cf. \cite[Section 4]{JGP08},
\begin{equation}\label{eqn:Weyl_tidal}
  \parbox{.57\textwidth}{a traceless self-adjoint operator $A\colon V\to V$, given by $A({\bf v}+\widetilde{\mathcal{D}}) = {\rm W}({\bf u},{\bf v}){\bf u} + \widetilde{\mathcal{D}}$, where ${\bf u}$ is any vector field on $\widetilde{M}$ such that $\tg({\bf u},{\bf w}) = 1$.}
\end{equation}
Clearly, ${\bf u}$ in \eqref{eqn:Weyl_tidal} is unique modulo $\widetilde{\mathcal{D}}^\perp$, so that $A(X+\widetilde{\mathcal{D}})$ is well-defined and, by \eqref{eqn:psi-q}, \hbox{$\gamma_\ast{\bf u} + \widetilde{\mathcal{D}} = q{\bf u}+\widetilde{\mathcal{D}}$} for every $\gamma \in \Gamma$. Every $\gamma \in \Gamma$ induces a linear isometry $C\colon V\to V$, acting via $C({\bf v}+\widetilde{\mathcal{D}}) = \gamma_\ast {\bf v} + \widetilde{\mathcal{D}}$, which leads to
\begin{equation}\label{eqn:homomorphism_C}
  \parbox{.8\textwidth}{a homomorphism $\Gamma\ni \gamma \mapsto C\in {\rm O}(V,\langle\cdot,\cdot\rangle)$ with $CAC^{-1}=q^2A$,}
\end{equation}$q$ being associated with $\gamma$ as in in (\ref{eqn:hom_qp_q}-b).

We will say that $(M,\g)$ itself is \emph{generic} if $A$ in \eqref{eqn:Weyl_tidal} is generic in the sense of Section \ref{sec:generic}. By \eqref{eqn:generic_m2}, $(M,\g)$ is always generic when $n=4$.


\section{The rank-one ECS models and their isometry groups}\label{sec:model}

Rank-one ECS models are built from the following data, cf. \cite{Roter74}:
\begin{equation}\label{eqn:ECS_data}
\parbox{.66\textwidth}{an integer $n\geq 4$, a pseudo-Euclidean vector space $(V,\langle\cdot,\cdot\rangle)$ of dimension $n-2$, a self-adjoint en\-do\-mor\-phism \hbox{$A \in \mathfrak{sl}(V)\smallsetminus \{0\}$}, and a nonconstant smooth function $f\colon I\to \R$ defined on an open interval $I\subseteq \R$.} 
\end{equation}

Then, defining $\kappa\colon I \times \R \times V \to \R$ by $\kappa(t,s,v) = f(t)\langle v,v\rangle + \langle Av,v\rangle$ and regarding $\langle\cdot,\cdot\rangle$ as a constant flat metric on $V$, we consider the simply connected $n$-dimensional pseudo-Riemannian manifold
\begin{equation}\label{eq:ECS_model}
  (\widehat{M},\hg) = \big(I\times \R\times V,\; \kappa\,{\rm d}t^2 + {\rm d}t\,{\rm d}s + \langle\cdot,\cdot\rangle\big),
\end{equation}
where we identify ${\rm d}t$, ${\rm d}s$ and $\langle\cdot,\cdot\rangle$ with their pull-backs to $\widehat{M}$.

By \cite[Theorem 4.1]{local_structure} $(\widehat{M},\hg)$ is a rank-one ECS manifold. Calling the manifolds \eqref{eq:ECS_model} \emph{models} is justified by their being locally \emph{universal}:
\begin{equation}\label{eqn:local_structure_cf}
  \parbox{.84\textwidth}{every point of a rank-one ECS manifold of dimension $n$ has a neighborhood isometric to an open subset of a manifold of type \eqref{eq:ECS_model}, with one possible exception in \eqref{eqn:ECS_data}: $\,f\,$ may be constant \cite[Theorem 4.1]{local_structure}.}
\end{equation}

Our two uses of the term `generic' are mutually consistent:
\begin{equation}\label{eqn:generic_consistency}
  \parbox{.77\textwidth}{genericity of $(\widehat{M},\hg)$ -- see the end of Section \ref{sec:comp_Dperp} -- is equivalent to that of the endomorphism $A$ in \eqref{eqn:ECS_data} as defined in Section \ref{sec:generic}.}
\end{equation}
Indeed, the Olszak distribution $\widehat{\mathcal{D}}$ of $(\widehat{M},\hg)$ -- defined in the Introduction -- is spanned by the null parallel coordinate vector field $\partial_s$ \cite[p. 93]{Roter74}, so that the leaves of $\widehat{\mathcal{D}}^\perp$ are the $\R\times V$ factor submanifolds of $\widehat{M}$ in \eqref{eq:ECS_model}. This allows us to isometrically identify $(V,\langle\cdot,\cdot\rangle)$ with the space of parallel sections of $\widehat{\mathcal{D}}^\perp\hskip-2pt/\widehat{\mathcal{D}}$, which, as shown in \cite[the lines following (6.3)]{DT_2}, also identifies $A$ in \eqref{eqn:ECS_data} with $A$ in \eqref{eqn:Weyl_tidal} (where one may set ${\bf u}=2\hskip1pt\partial_t$).

Central to the discussion are: the $2(n\hskip-1.5pt-\hskip-1.5pt 2)$-dimensional \hbox{symplectic vector space}
\begin{equation}\label{eqn:E}
\parbox{.84\textwidth}{$(\mathcal{E},\Omega)$ consisting of all solutions \hbox{$u\colon I\to V$} of the second-order equation \hbox{$\ddot{u} = fu+Au$}, where $\Omega$ is defined by \hbox{$\Omega(u,w) = \langle \dot{u},w\rangle - \langle u,\dot{w}\rangle$}, and its associated Heisenberg group $\mathrm{H}$: the Cartesian product $\R\times\mathcal{E}$ with the operation defined by $\,(r,u)(\hat{r},\hat{u}) = (r+\hat{r} - \Omega(u,\hat{u})\,,\, u+\hat{u})$.}
\end{equation} We also need \begin{equation}\label{eqn:group_Q}
\parbox{.84\textwidth}{the subgroup $\mathrm{S}$ of ${\rm Aff}(\R) \times {\rm O}(V,\langle\cdot,\cdot\rangle)$ formed by all $(q,p,C)$ having $CAC^{-1}\,=\,q^2A$, with $\,qt+p\in I\,$ and $f(t) = q^2f(qt+p)$ for all $t\in I$.}
\end{equation}

Each of $q$, $(q,p)$, and $C$ depends homomorphically on $\sigma = (q,p,C)$, so that $\mathrm{S}$ acts (from the left) on $I$, $\R$, and ${\rm C}^\infty(I,V)$ via, respectively,
\begin{equation}
  \label{eq:H_actions}
 \parbox{.8\textwidth}{{\rm i)} \,$\sigma t = qt+p$, \quad {\rm ii)}\, $\sigma s = q^{-1}s$, \quad {\rm iii)}\, $(\sigma u)(t) = Cu(q^{-1}(t-p))$.}
\end{equation}
As the notations in (\ref{eq:H_actions}-i) and (\ref{eq:H_actions}-ii) are in conflict, we will adopt only the former and explicitly write $q^{-1}s$ for the latter, always understanding that $q$ is the first component of $\sigma$. The action of $\mathrm{S}$ on ${\rm C}^\infty(I,V)$, obviously leaving $\mathcal{E}$ in \eqref{eqn:E} invariant, restricts to an action on $\mathcal{E}$ with $\det \sigma = q^{2-n}$ on $\mathcal{E}$ for all $\sigma\in \mathrm{S}$, since
\begin{equation}\label{eqn:almost_symp}
  \parbox{.68\textwidth}{$\sigma^*\Omega = q^{-1}\Omega$, \,if $\,\sigma\,$ is regarded as an operator $\sigma\colon \mathcal{E}\to \mathcal{E}$.}
\end{equation}

Rephrasing \cite[Theorem 3.1]{DT_3}, we have:

\begin{teo}\label{teo:iso_heis} The isometry group of a model $(\widehat{M},\hg)$, with \eqref{eqn:ECS_data}--\eqref{eq:ECS_model}, can be identified with the set $\mathrm{S} \times \mathrm{H}$, cf. \eqref{eqn:E}--\eqref{eqn:group_Q}, so that $\varPhi = (\sigma,r,u)$ acts on $(\widehat{M},\hg)$ via 
\[ 
\varPhi(t,s,v) = \big(\sigma t\,,\, -\langle \dot{u}(\sigma t), 2\sigma v + u(\sigma t)\rangle + q^{-1} s+r\,,\, \sigma v + u(\sigma t)\big),
\] for every $(t,s,v) \in \widehat{M}$, and the group operation in $\mathrm{S}\times\mathrm{H}$ becomes \[  (\sigma,r,u)(\widehat{\sigma},\widehat{r},\widehat{u}) = \big(\sigma\widehat{\sigma}\,,\,r+q^{-1}\widehat{r} - \Omega(u,\sigma\widehat{u})\,,\, u+\sigma\widehat{u}\big),  \]for $(\sigma,r,u),(\widehat{\sigma},\widehat{r},\widehat{u})\in\mathrm{S}\times\mathrm{H}$. Thus, ${\rm Iso}(\widehat{M},\hg)$ is isomorphic to a semidirect product \hbox{$\mathrm{S} \ltimes \mathrm{H}$}, where the diagonal action of $\mathrm{S}$ on $\mathrm{H}$ is defined via \eqref{eq:H_actions}: $\sigma\cdot (r,u) = (q^{-1} r, \sigma u)$.
\end{teo}

\begin{obs}\label{rem:identify_H}
  We identify $\mathrm{H}$ with the normal subgroup $\{(1,0,{\rm Id})\} \times \mathrm{H}$ of ${\rm Iso}(\widehat{M},\hg)$, the kernel of the homomorphism ${\rm Iso}(\widehat{M},\hg) \ni (\sigma,r,u) \mapsto \sigma \in \mathrm{S}$.
\end{obs}

\begin{obs}\label{obs:tg_leaves}
In any rank-one ECS manifold, the leaves of $\mathcal{D}^\perp$ are totally geodesic, $\mathcal{D}^\perp$ being parallel. In addition, the resulting induced connection on each leaf is \emph{flat}. Namely, \eqref{eqn:local_structure_cf} allows us to assume that the manifold has the form \eqref{eqn:ECS_data}, with \eqref{eq:ECS_model} except for nonconstancy of $f$. For $i,j$ ranging over $2,\ldots,n-1$, any linear coordinates $x^i$ on $V$ form, along with $x^1=t\,$ on $I$ and $x^n=s/2$ on $\R$, a coordinate system on $I\times\R\times V$, and then -- see \cite[the lines following formula (6.2)]{DT_2} -- the coordinate vector fields $\partial_n$ and $\partial_i$ span $\mathcal{D}^\perp$, while, according to \cite[p. 93]{Roter74}, $\,\varGamma_{ij}^{\bullet}=\varGamma_{in}^{\bullet} =\varGamma_{nn}^{\bullet}=0$, where $\bullet$ denotes any index.
\end{obs}

\begin{obs}\label{rem:conclusions_E}
  In a rank-one ECS model manifold $(\widehat{M},\hg)$, let there exist a subgroup $\Gamma$ of ${\rm Iso}(\widehat{M},\hg)$ acting freely and properly discontinuously on $(\widehat{M},\hg)$ with a compact isometric quotient $M = \widehat{M}/\Gamma$. Using \eqref{eqn:passing_to} we also assume that $q>0$ whenever $(q,p,C,r,u) \in \Gamma$. If the resulting ECS manifold $(M,\g)$ is translational, then all $(q,p,C,r,u) \in \Gamma$ have $q=1$, due to \cite[formula (2.5-a)]{DT_3}. Also, $I=\R$ in \eqref{eqn:ECS_data}:  otherwise, all $(1,p,C,r,u)\in \Gamma$ acting on $\widehat{M}$ (see Theorem \ref{teo:iso_heis}) would have $p=0$ (since $t\mapsto t+p$ sends $I$ onto itself), and so
\begin{equation}\label{eqn:t_descends}
  \parbox{.52\textwidth}{$t$ would descend to a function without critical points on the compact manifold $M$.}
\end{equation}
Finally, according to \cite[formula (3.1)]{DT_1}, such translational $(M,\g)$ is geodesically complete but not locally homogeneous.
\end{obs}

\section{Proofs of Theorems \ref{teoA:generic_mc}, \ref{teo:generic_Dperp}, and \ref{teoA:uc_model}}\label{sec:main_proofs}

In the first two proofs below, $(\widetilde{M},\tg)$ is the isometric universal covering manifold of the $n$-dimensional compact rank-one ECS manifold $(M,\g)$ in question. The objects $\pi$, $\Gamma$, $\widetilde{\mathcal{D}}$, ${\bf w}$, $t$, $I$, $(V,\langle\cdot,\cdot\rangle)$, and $A$ are all defined as in Section \ref{sec:comp_Dperp}.

Recall from the Introduction that $(\widetilde{M},\tg)$ is said to be $\widetilde{\mathcal{D}}^\perp\!$-complete if all the leaves of $\widetilde{\mathcal{D}}^\perp$ are complete, while its maximal completeness means that every non-com\-ple\-te maximal geodesic intersects all leaves of $\widetilde{\mathcal{D}}^\perp$. As $\widetilde{\mathcal{D}}^\perp$ and geodesics in $(\widetilde{M},\tg)$ are mapped under $\pi$ onto their analogs in $(M,\g)$, it suffices to prove Theorem \ref{teo:generic_Dperp} for $(\widetilde{M},\tg)$ rather than $(M,\g)$.

\begin{proof}[Proof of Theorem \ref{teo:generic_Dperp}]
  We apply Lemma \ref{lem:abstract_completeness} to any given leaf $L$ of $\widetilde{\mathcal{D}}^\perp$ with the flat connection induced on it (Remark \ref{obs:tg_leaves}), setting $\mathcal{X} = \R {\bf w}$ and \hbox{$\mathcal{Y} = \bigoplus_{j=1}^m\R\widetilde{\vec{y}}_j$} for $m=n-2$ vector fields $\widetilde{\vec{y}}_j$ on $\widetilde{M}$ defined below, which, due to their \hbox{$\Gamma$-invariance}, will descend to the compact manifold $M=\widetilde{M}/\Gamma$, making each ${\bf v} \in \mathcal{Y}$ complete.

We begin by assuming \eqref{eqn:Dperp_tr_orient} and choosing a basis $({\bf v}_1,\ldots,{\bf v}_m)$ of $V$. Let \hbox{$K \subseteq (0,\infty)$} be the image of $\Gamma$ under (\ref{eqn:hom_qp_q}-b). By \eqref{eqn:Dperp_tr_orient}, $K$ is either trivial, or infinite. In the former case, the image of \eqref{eqn:homomorphism_C} is finite due to genericity of $A$, and so its kernel has finite index in $\Gamma$. Thus, \eqref{eqn:passing_to} allows us to further assume that $\Gamma$ acts trivially on $V$, via \eqref{eqn:homomorphism_C}, and
$({\bf v}_1,\ldots,{\bf v}_m)$ can be completely arbitrary. If $K$ is infinite, \eqref{eqn:homomorphism_C} implies nilpotency of $A$, and we select a basis $({\bf v}_1,\ldots,{\bf v}_m)$ of $V$ for which \eqref{eqn:fit_A_g}--\eqref{eqn:two_isometries} holds. As $\delta$ in \eqref{eqn:two_isometries} depends homomorphically on $\gamma$, using \eqref{eqn:passing_to} we may require that $\delta = 1$ for every $\gamma \in \Gamma$.

In either case, we fix a Riemannian metric $\g^\circ$ on $M$ and lift $({\bf v}_1,\ldots,{\bf v}_m)$ to vector fields $({\boldsymbol y}_1,\ldots, {\boldsymbol y}_m)$ tangent to leaves of $\widetilde{\mathcal{D}}^\perp$ which are $\pi^*\mathtt{g}^\circ$-orthogonal to ${\bf w}$. For $\psi$ as in (\ref{eqn:psi-dt}-a), the vector fields $\widetilde{\vec{y}}_j = \psi^{2j-1-m}{\boldsymbol y}_j$, are \hbox{$\Gamma$-invariant} (cf. \eqref{eqn:two_isometries} and \eqref{eqn:psi-q}), which completes the proof. (When $K=\{1\}$, we may set $\psi=1$.)
\end{proof}

\begin{proof}[Proof of Theorem \ref{teoA:generic_mc}]
In view of \eqref{eqn:t-levels} and $\widetilde{\mathcal{D}}^\perp\!$-completeness of $(\widetilde{M},\tg)$ (due to Theorem \ref{teo:generic_Dperp}), every maximal geodesic of $(\widetilde{M},\tg)$ transverse to $\widetilde{\mathcal{D}}^\perp$ can be pa\-ra\-me\-tri\-zed by $t$, as $t\colon\widetilde{M}\to I$ restricted to its image is a diffeomorphism onto a subinterval $I'\subseteq I$, cf. \cite[Remark 7.2]{JGP08}. To show that $I'=I$, we invoke \cite[Lemma 7.3]{JGP08} with minimal modifications in its proof, which never uses the assumption stated there that the metric under consideration is Lorentzian. See Appendix \ref{app:technical} for details.
\end{proof}

\begin{proof}[Proof of Theorem \ref{teoA:uc_model}]
  Assume that $(\widetilde{M},\tg)$ is a $n$-dimensional simply connected and maximally complete rank-one ECS manifold, and again choose  $\pi$, $\Gamma$, $\widetilde{\mathcal{D}}$, ${\bf w}$, $t$, $I$, $(V,\langle\cdot,\cdot\rangle)$, and $A$ as in Section \ref{sec:comp_Dperp}. For the function $f\colon I\to \R$ such that $\widetilde{\rm Ric} = (2-n)f(t)\,{\rm d}t\otimes {\rm d}t$ (cf. \cite[formula (5.6)]{DT_2}), we consider the model ECS manifold $(\widehat{M},\hg)$ built from these ingredients as in \eqref{eq:ECS_model}. An isometry between $(\widehat{M},\hg)$ and $(\widetilde{M},\tg)$ is defined as in the proof of \cite[Theorem 7.1]{JGP08}, with no significant changes. Details are given in Appendix \ref{app:model}.
\end{proof}

\section{The transitive-commutation property}\label{sec:TCP}

For a group-subgroup pair $(G,H)$, consider the \emph{transitive-commutation property}: the commutation relation on $G\smallsetminus H$ is transitive, that is, the equalities $xy = yx$ and $yz = zy$ for $x,y,z \in G \smallsetminus H$ imply $xz = zx$.

\begin{lem}\label{lem:TCP_equiv}Let $(G,H)$ be a group-subgroup pair.
  \begin{enumerate}[\normalfont(a)]
  \item $(G,H)$ has the transitive-commutation property if and only if there exists a family $\mathcal{K}$ of Abelian subgroups of $G$ such that $\{K\smallsetminus H\}_{K\in\mathcal{K}}$ is a partition of $G\smallsetminus H$, and any two elements of $G\smallsetminus H$ which commute lie in the same $K\in\mathcal{K}$. 

  \item Whenever $(G,H)$ has the transitive-commutation property and $\mathcal{K}$ is as in \emph{(a)}, any Abelian subgroup of $G$ is contained in $H$, or in a unique element of $\mathcal{K}$.
  \end{enumerate}
\end{lem}


\begin{proof}
The `only if' part of (a) is obvious. Conversely, since commutation is now an equivalence relation on $G\smallsetminus H$, the subgroup $K$ generated by any equivalence class $E \subseteq G \smallsetminus H$ of the commutation relation is Abelian and $K \smallsetminus H = E$, the right-to-left inclusion being immediate, the other one clear as elements in $K \smallsetminus H$ commute with all of $E$ and hence lie in $E$. This yields (a).
  
 To prove (b), let $N$ be an Abelian subgroup of $G$. If there exists an element $x\in N\smallsetminus H$, we fix the unique subgroup $K\in \mathcal{K}$ with $N\smallsetminus H \subseteq K$. Then we must have $N\subseteq K$, for if $y\in (N\cap H)\smallsetminus K$, then $xy\in N\smallsetminus H$ and so \hbox{$y = x^{-1}(xy) \in K$.}
\end{proof}

  

\begin{obs}
  One easily verifies the fact (not needed in our argument) that $K$ associated with $E$ in the above proof is both the unique Abelian subgroup of $G$ containing $E$, and the centralizer of $E$.
\end{obs}

\section{Generic homogeneous models}\label{sec:LH_details}

By a \emph{standard homogeneous rank-one ECS model manifold} (briefly, a \emph{standard homogeneous model}), we mean $(\widehat{M},\hg)$ as in \eqref{eqn:ECS_data}--\eqref{eq:ECS_model} with
\begin{equation}\label{eqn:f_LH}
  \parbox{.9\textwidth}{$I = (0,\infty)$ and $f(t) = \dfrac{c^2 - 1/4}{t^2}$, where $c\in [0,1/2)\cup (1/2,\infty) \cup {\rm i}(0,\infty)$.}
\end{equation}All such $(\widehat{M},\hg)$ are homogeneous, as pointed out in \cite[Remark 1 on p. 172]{derdzinski-cm-78}. (This also follows from Theorem \ref{teo:iso_heis}: one easily sees that the group of all elements $(q,p,C,r,u) \in {\rm Iso}(\widehat{M},\hg)$ having $p=0$ acts on $\widehat{M}$ transitively.) The factor $c^2-1/4$ is just any real constant $h\neq 0$, written so for later convenience.

\begin{lem}\label{lem:open_sub}
Every locally homogeneous rank-one ECS model manifold is isometric to an open submanifold $(a,b)\times \R\times V$ of a standard homogeneous model \eqref{eq:ECS_model}.  
\end{lem}

\begin{proof}
  By \cite[formula (3.4)]{DT_1}, in \eqref{eqn:ECS_data}, $f\neq 0$ everywhere and $|f|^{-1/2}$ is a linear function of $t$, so that \hbox{$f(t)=h(t-t_0)^{-2}$} for some real $h\neq 0$ and $t_0$. Under a suitable coordinate change $(t,s)\mapsto (qt+p,q^{-1}s)$, with \hbox{$(q,p)\in {\rm Aff}(\R)$}, \eqref{eqn:ECS_data}-\eqref{eq:ECS_model} remain valid, allowing us to assume that $t_0=0$ and $I\subseteq (0,\infty)$.
\end{proof}

As a trivial consequence of Lemma \ref{lem:open_sub} and \eqref{eqn:local_structure_cf}, we obtain:

\begin{cor}
  All locally homogeneous rank-one ECS manifolds are locally isometric to standard homogeneous models. 
\end{cor}

If a standard homogeneous model $(\widehat{M},\hg)$ is also generic, cf. \eqref{eqn:generic_consistency}, then, with notations of \eqref{eqn:E} and Theorem \ref{teo:iso_heis}, elements of
\begin{equation}\label{eq:identity_comp}
  \parbox{.5\textwidth}{the identity component ${\rm G}_0$ of ${\rm Iso}(\widehat{M},\hg)$}
\end{equation}
have the form $(q,0,C_q,r,u)$, where $q\in (0,\infty)$, $(r,u)\in\mathrm{H}$, cf. \eqref{eqn:E} and Remark \ref{rem:identify_H}, while $C_q\colon V\to V$ is as in \eqref{eqn:two_isometries} with $\delta=1$. We also consider
\begin{equation}\label{defn_sigma_q}
  \parbox{.7\textwidth}{$\sigma_q\colon \mathcal{E}\to \mathcal{E}$ defined as in (\ref{eq:H_actions}-iii): $(\sigma_qu)(t) = C_q u(t/q)$.}
\end{equation}
Abbreviating $\varPhi = (q,0,C_q,r,u)$ simply to $(q,r,u)$, we may now
\begin{equation}\label{idenfity_G0_interval}
  \parbox{.4\textwidth}{identify ${\rm G}_0$ with $(0,\infty)\times \mathrm{H}$,}
\end{equation}and, for $\widehat{\varPhi}\in {\rm G}_0$ and $(t,s,v)\in\widehat{M}$, we obtain, from Theorem \ref{teo:iso_heis} that
\begin{equation}\label{eqn:operations_LH}
  \parbox{.915\textwidth}{
    \begin{enumerate}[\normalfont i)]
    \item $\varPhi\widehat{\varPhi} = (q\widehat{q}, r+q^{-1}\widehat{r} - \Omega(u,\sigma_q\widehat{u}), u+\sigma_q\widehat{u})$,
    \item $\varPhi^{-1}= (q^{-1}, -qr, -\sigma_q^{-1}u)$,
    \item the $\mathcal{E}$-component of the commutator $[\varPhi,\widehat{\varPhi}]$ is $(\sigma_q-1)\widehat{u} - (\sigma_{\widehat{q}}-1)u$,
    \item $\varPhi(t,s,v) = \big(qt\,,\, -\langle\dot{u}(qt), 2C_qv+u(qt)\rangle + q^{-1}s+r\,,\, C_qv+u(qt)\big)$
    \end{enumerate}}
\end{equation} With $1$ denoting the identity operator $\mathcal{E}\to\mathcal{E}$, conjugation by $\varPhi$ has the form
\[\varPhi\widehat{\varPhi}\varPhi^{-1}\! =\!\left(\widehat{q}\,,\, q^{-1}\widehat{r}\!+\!(1\!-\!\widehat{q}^{-1})r \!-\! \Omega((1+\sigma_{\widehat{q}})u, \sigma_q\widehat{u}) \!+\! \Omega(u,\sigma_{\widehat{q}}u)\,,\, (1\!-\!\sigma_{\widehat{q}})u\!+\!\sigma_q\widehat{u}\right)\!.\]

Let us now fix a generic homogeneous model $(\widehat{M},\hg)$.

By \cite[Theorem 5.1]{DT_3}, each $\sigma_q\colon \mathcal{E}\to \mathcal{E}$ has the spectrum $\lambda_j^\pm = q^{m+1-2j}\mu^\pm$, $j=1,\ldots, m$, for the eigenvalues $\mu^\pm \in \C$ of the operator $T$ with $(Tu)(t) = u(t/q)$, on the space $\mathcal{W}$ of solutions $y\colon (0,\infty) \to \C$ to the ordinary differential equation $\ddot{y}=fy$. The expression for $f$ in \eqref{eqn:f_LH} gives $\mu^\pm = q^{-\frac{1}{2} \mp c}$, the corresponding $T$-diagonalizing (or, $T$-triangular) basis of $\mathcal{W}$ being $t\mapsto t^{\frac{1}{2} \pm c}$ if $c\neq 0$ (or, respectively, $t\mapsto t^{\frac{1}{2}}$ and $t\mapsto t^{\frac{1}{2}}\log t$ when $c=0$). Hence
\begin{equation}\label{eqn:spectrum_sigmaq}
  \parbox{.75\textwidth}{the spectrum of $\sigma_q$ becomes  $\lambda_j^\pm = q^{m+\frac{1}{2} - 2j \mp c}$, for $j=1,\ldots, m$.}
\end{equation}

The assignment $(0,\infty)\ni q\mapsto\sigma_q \in\mathrm{GL}(\mathcal{E})$, being a homo\-mor\-phism, has an infinitesimal generator
$B\in\mathfrak{gl}(\mathcal{E})$, with
$\sigma_q=\exp[(\log q)B]$. By \eqref{eqn:spec_B_general} and \eqref{eqn:spectrum_sigmaq},
\begin{equation}\label{eqn:specB} \parbox{.74\textwidth}{the spectrum of $B$ is $\kappa_j^\pm = m+\dfrac{1}{2} - 2j \mp c$, where $j=1\ldots, m$.}\end{equation}

\begin{lem}\label{lem:E0andE+} For $\mathcal{E}$ and $B$ as above, let $\mathcal{E}_0 = \ker B$ and $\mathcal{E}_+ = B(\mathcal{E})$. Then:
  \begin{enumerate}[\normalfont (a)]
  \item Either $\mathcal{E}_0$ is trivial, or $\dim\mathcal{E}_0=1$ 
  and $B$ is di\-ag\-o\-nal\-iz\-able with $2m$ distinct real eigenvalues. In both cases, $\mathcal{E} = \mathcal{E}_0 \oplus \mathcal{E}_+$ and $(\sigma_q -1)(\mathcal{E}_0) = \{0\}$ if $q\in(0,\infty)$.  
  \item For every $q\in (0,\infty)\smallsetminus \{1\}$ the operator $\sigma_q-1\colon \mathcal{E}_+\to\mathcal{E}_+$ is an isomorphism.
  \end{enumerate}
\end{lem}

\begin{proof}
If $\det B=0$, some $\kappa_j^\pm$ in \eqref{eqn:specB} equals $0$, so that 
$2c=\pm(2m-4j+1)$ is an odd integer. Hence, the eigenvalues of $B$ are real and mutually distinct: $j\mapsto\kappa_j^\pm$ is injective for a fixed sign $\pm$, 
while $2c$ would be the \emph{even} integer $2(i-j)$ if there existed $i$ and $j$ with $\kappa_j^+=\kappa_i^-$. Since $\sigma_q = \exp[(\log q)B]$, this yields (a).

To prove (b), we use \eqref{eqn:spectrum_sigmaq} and consider two cases. If $c\in \R$, the resulting injectivity of $a\mapsto q^a$ on $\R$ for $q\neq 1$ gives $\lambda_j^\pm \neq 1$ except -- see the last paragraph -- when $\dim\mathcal{E}_0=1$ and $(j,\pm)$ is the unique pair with $\kappa_j^\pm=0$. If $c\in {\rm i}(0,\infty)$, $\lvert\lambda_j^\pm\rvert = q^{m-2j+1/2}$, being a half-integer power of $q$, cannot equal $1$ unless $q=1$, and hence $\lambda_j^\pm \neq 1$ again.
\end{proof}

\begin{lem}\label{teo:TCP_iso}
 For a generic standard homogeneous model, the group-subgroup pair $({\rm G}_0,{\rm H})$ given by \eqref{eq:identity_comp} and \eqref{eqn:E} has the transitive-commutation property of Section \ref{sec:TCP}, and its equivalence classes generate subgroups of ${\rm G}_0$ acting freely on $\widehat{M}$, and isomorphic to $\R^k$, for $k=\dim\mathcal{E}_0+1\in \{1,2\}$.
\end{lem}

\begin{proof}
 Define $J\colon\R\times\mathcal{E}_+\times(0,\infty)\times\mathcal{E}_0\to{\rm G}_0\smallsetminus \mathrm{H}$ by
\[
J(a,z,q,w)=\left(q\,,\,a(1-q^{-1}) + \Omega(z,\sigma_q z+(1+q^{-1})w)\,,\,(\sigma_q-1)z+w\right).
\]
It is an obvious consequence of Lemma \ref{lem:E0andE+}, (\ref{eqn:operations_LH}-i) and \eqref{eqn:almost_symp} that
\begin{enumerate}[i)]
\item $J$ maps $\R\times\mathcal{E}_+\times[(0,\infty)\smallsetminus\{1\}]\times\mathcal{E}_0$ 
bijectively onto $\mathrm{G}_0\smallsetminus \mathrm{H}$,
\item $J(a,z,\,\cdot\,,\,\cdot\,)\colon (0,\infty)\times\mathcal{E}_0\to{\rm G}_0$ is an injective group homo\-mor\-phism whenever $(a,z)\in\R\times\mathcal{E}_+$.
\end{enumerate}    
By (ii), the images $K_{a,z}$ of $J(a,z,\,\cdot\,,\,\cdot\,)$, with
$(a,u)\in\R\times\mathcal{E}_+$, are connected Abel\-i\-an Lie sub\-groups of
${\rm G}_0$, isomorphic to $\R$ (if $\mathcal{E}_0=\{0\}$) or $\R^2$ (when $\dim\mathcal{E}_0=1$), while, due to (i), the family $\{K_{a,z}\smallsetminus\mathrm{H}:(a,z)\in\R\times\mathcal{E}_+\}$ is a partition of ${\rm G}_0\smallsetminus\mathrm{H}$.

We now define $F:{\rm G}_0\smallsetminus\mathrm{H}\to\R\times\mathcal{E}_+$ by
$F(J(a,z,q,w))=(a,z)$ when $q\neq 1$, which makes sense according to (i).
Thus, $F$ associates with $\varPhi\in{\rm G}_0\smallsetminus\mathrm{H}$ the
unique $(a,z)$ for which $\varPhi\in K_{a,z}$. It follows that
\begin{enumerate}[i)]\setcounter{enumi}{2}
\item $\varPhi,\widehat{\varPhi}\in {\rm G}_0\smallsetminus\mathrm{H}$ commute if and only if $F(\varPhi)=F(\widehat{\varPhi})$.
\end{enumerate}    
In fact, the \enquote*{if} part is obvious as $\varPhi$ and $\widehat{\varPhi}$ then lie in the same Abel\-i\-an sub\-group of ${\rm G}_0$. By (\ref{eqn:operations_LH}-i), two elements $(q,r,u),(\widehat{q},\widehat{r},\widehat{u}) \in {\rm G}_0$ commute if and only if
\begin{equation}\label{eqn:commutation_generic}
  r+q^{-1}\widehat{r} - \Omega(u,\sigma_q\widehat{u}) = \widehat{r}+\widehat{q}^{-1}r-\Omega(\widehat{u},\sigma_{\widehat{q}}u)\quad\mbox{and}\quad (\sigma_q-1)\widehat{u} = (\sigma_{\widehat{q}}-1)u.
\end{equation}We now prove the `only if' part of (iii) assuming -- see (i) -- that $J(a,z,q,w)$ commutes with $J(\widehat{a},\widehat{z},\widehat{q},\widehat{w})$ and $q\neq 1\neq \widehat{q}$. The second equality of \eqref{eqn:commutation_generic} with $(u,\widehat{u})=((\sigma_q-1)z+w,(\sigma_{\widehat{q}}-1)\widehat{z}+\widehat{w})$ yields $z=\widehat{z}$, since -- by Lemma \ref{lem:E0andE+}(b) -- the operators $\sigma_q-1$ annihilate $\mathcal{E}_0$, and form a family of mutually commuting auto\-mor\-phisms when restricted to $\mathcal{E}_+$ (the latter since $q\mapsto\sigma_q$ is a homo\-mor\-phism). As $z=\widehat{z}$, the first equality of \eqref{eqn:commutation_generic}, for $(u,\widehat{u})$ chosen above, and $(r,\widehat r)$ replaced by the $\R$-components of $J(a,z,q,w)$ and $J(\widehat{a},\widehat{z},\widehat{q},\widehat{w})$, easily yields $a=\widehat{a}$, and (iii) follows. The conclusion is now immediate from Lemma \ref{lem:TCP_equiv}(a).
\end{proof}

\section{Proof of Theorem \ref{teoA:gen_cpt_translational}}\label{sec:proof_E}

We fix a generic compact rank-one ECS manifold $(M,\g)$. By Corollary \ref{teoA:uc_cpct_model}, the universal covering of $(M,\g)$ is isometrically identified with a model $(\widehat{M},\hg)$ as in \eqref{eqn:ECS_data}--\eqref{eq:ECS_model}, and then $\widehat{M}/\Gamma = M$, as at the beginning of Section \ref{sec:comp_Dperp}, where \eqref{eqn:passing_to} also allows us to assume that $\Gamma \subseteq {\rm G}_0$, cf. \eqref{eq:identity_comp}.

By Remark \ref{rem:conclusions_E}, the ``translational'' conclusion about $(M,\g)$, which we prove in the subsequent paragraphs, implies the other two assertions of Theorem \ref{teoA:gen_cpt_translational}.

Next, suppose that, on the contrary, our $(M,\g)$ is not translational. The dichotomy \eqref{eqn:dichotomy} in the Introduction and \cite[Corollary D]{DT_3} imply that $(\widehat{M},\hg)$ is locally homogeneous. Lemma \ref{lem:open_sub} now allows us to write $\widehat{M} = (a,b)\times \R\times V$. Then, $(a,b)=(0,\infty)$, or else all elements $\varPhi = (q,r,u)\in\Gamma$, acting on $\widehat{M}$ via (\ref{eqn:operations_LH}-iv), would have $q=1$ (as $t\mapsto qt$ maps $(a,b)$ onto itself), leading to \eqref{eqn:t_descends}.

  The group $\Sigma = \Gamma \cap \mathrm{H}$, cf. Remark \ref{rem:identify_H}, is the kernel of the homomorphism
  \begin{equation}\label{eq:homomorphism_phi}
    \Gamma\ni (q,r,u) \mapsto q\in (0,\infty).
  \end{equation}
As an immediate consequence of \cite[Lemma 2.2(b) and (f) in Section 3]{DT_3} and \cite[Theorems A, B, and Lemma 2.1]{DT_3},\begin{equation}\label{eqn:trivial_int_flow}
  \parbox{.855\textwidth}{$\Sigma \cap (\{1\}\times \R\times \{0\})$ is trivial and the image of \eqref{eq:homomorphism_phi} is dense in $(0,\infty)$.}
\end{equation}
The last line in \eqref{eqn:E} now has three consequences. First, $\Sigma \ni (1,r,u)\mapsto u\in\mathcal{E}$ is a homomorphism, and its injectivity due to \eqref{eqn:trivial_int_flow} implies that $\Sigma$ is Abelian. Secondly, the image $\Lambda$ of this last homomorphism spans a subspace $\mathcal{L}$ of $(\mathcal{E},\Omega)$ which is isotropic, in the sense that $\Omega = 0$ on it. Finally, $\R \times \mathcal{L}$ is an Abelian subgroup of $\mathrm{H}$ (see Remark \ref{rem:identify_H}) containing $\Sigma$, and its group operation coincides with the addition in the vector space $\R \times \mathcal{L}$. Applying Remark \ref{rem:comp_fibrations} to $\Sigma$ rather than $\Gamma$ and the subspace $\mathcal{Z}$ of $\R\times\mathcal{L}$ spanned by $\Sigma$, we see that
\begin{equation}\label{eqn:dichotomy_lattice} \parbox{.565\textwidth}{$\Sigma$ is a lattice in $\mathcal{Z}$ and either $\mathcal{Z} = \R\times\mathcal{L}$, or $\mathcal{Z}\,$ is a hyperplane in $\,\R\times\mathcal{L}\,$ transverse to $\,\R\times\{0\}$.}\end{equation}
Any $(q,r,u)\in \Gamma$ leads to the conjugation mapping ${\rm C}_{q,r,u}\colon {\rm H}\to {\rm H}$ given by
\begin{equation}\label{defn_Cqru}
 {\rm C}_{q,r,u}(\widehat{r},\widehat{u}) = \left(q^{-1}\widehat{r}-2\Omega(u,\sigma_q\widehat{u}), \sigma_q\widehat{u}\right),
\end{equation}
cf. Remark \ref{rem:identify_H} and the lines after \eqref{eqn:operations_LH} with $\widehat{q}=1$. This makes ${\rm C}_{q,r,u}$ a linear en\-do\-mor\-phism of $\R\times\mathcal{E}$ and, by \eqref{eqn:spectrum_sigmaq}, for $\Lambda$ and $\mathcal{Z}$ as in the lines before \eqref{eqn:dichotomy_lattice},
\begin{equation}\label{eqn:spectrum_Cqru} \parbox{.74\textwidth}{the spectrum of ${\rm C}_{q,r,u}$ consists of $q^{-1}$ and the spectrum of $\sigma_q$, while $\,{\rm C}_{q,r,u}(\Sigma)=\Sigma\,$, so that $\,\sigma_q(\Lambda)=\Lambda\,$ and $\,{\rm C}_{q,r,u}(\mathcal{Z})=\mathcal{Z}\,$.}\end{equation}
From \eqref{eqn:spectrum_Cqru}, due to \eqref{idenfity_G0_interval}, \eqref{defn_sigma_q}, and the denseness conclusion in \eqref{eqn:trivial_int_flow},
\begin{equation}
 \parbox{.58\textwidth}{$\sigma_q(\Lambda) = \Lambda$ and $\sigma_q(\mathcal{L}) = \mathcal{L}$ for every $q \in (0,\infty)$.}
\end{equation}
It follows that ${\rm rank}\,\Sigma \leq 1$. Namely, by \eqref{eqn:spectrum_Cqru} and \eqref{defn_Cqru}, each \hbox{${\rm C}_{q,r,u} \in \mathfrak{gl}(\R \times \mathcal{E})$} leaves invariant the subspaces $\mathcal{Z}$ and $\R \times \{0\}$, as well as $\mathcal{Z}' = \mathcal{Z} \cap (\R \times \{0\})$, so that it descends to a linear endomorphism of the quotient $\mathcal{Z}/\mathcal{Z}'$ and, in view of \eqref{defn_Cqru}, the isomorphism $\mathcal{Z}/\mathcal{Z}' \to \mathcal{L}$ induced by the projection $(r,u) \mapsto u$ makes the latter endomorphism correspond to $\zeta_q:\mathcal{L} \to \mathcal{L}$ arising as the restriction of $\sigma_q$ to $\mathcal{L}$. The spectrum of ${\rm C}_{q,r,u}$ acting on $\mathcal{Z}$ thus equals the spectrum of $\zeta_q$ in $\mathcal{L}$, augmented -- only when $\mathcal{Z} = \R \times \mathcal{L}$ in \eqref{eqn:dichotomy_lattice} -- by the eigenvalue $q^{-1}$. At the same time, the infinitesimal generator of the Lie-group homomorphism $(0,\infty) \ni q \mapsto \zeta_q\in {\rm GL}(\mathcal{L})$, cf. \eqref{defn_Cqru}, being the restriction to $\mathcal{L}$ of $B$ appearing in \eqref{eqn:specB}, has the spectrum $\alpha(1),\ldots,\alpha(k)$ which is a part of that in \eqref{eqn:specB}. For reasons stated three lines after \eqref{eqn:spec_B_general}, $\zeta_q$ then must have spectrum $q^{\alpha(1)},\ldots,q^{\alpha(k)}$. Thus, the spectrum of $C_{q,r,u}:\mathcal{Z} \to \mathcal{Z}$ consists of
\begin{equation}\label{eqn:restricted_spectrum}
  \parbox{.4\textwidth}{$q^{\alpha(1)},\ldots,q^{\alpha(k)}$ and, possibly, $q^{-1}$,}
\end{equation}
which are complex powers of $q$ with exponents not depending on $q$, so that, as ${\rm C}_{q,r,u}(\Sigma) = \Sigma$ in
\eqref{eqn:spectrum_Cqru}, the trace of ${\rm C}_{q,r,u}:\mathcal{Z} \to \mathcal{Z}$, being integer-valued and continuous in $q$, must be constant. By \eqref{eqn:sum_powers_cte}, the eigenvalues \eqref{eqn:restricted_spectrum} are all equal to $1$, which excludes $q^{-1}$ and, due to continuity in $q$, gives $\alpha(1)=\ldots=\alpha(k) = 0$. From Lemma \ref{lem:E0andE+}(a), it now follows that $k = {\rm rank}\,\Sigma \leq 1$.

  However, the conclusion that ${\rm rank}\,\Sigma\leq 1$ further implies that $\Gamma$ is Abelian. Indeed, when $\Sigma$ is trivial, injectivity of \eqref{eq:homomorphism_phi} makes $\Gamma$ isomorphic to a subgroup of $(0,\infty)$, while if ${\rm rank}\,\Sigma = 1$, we fix a generator $(1,b,w)$ of $\Sigma$, noting that \eqref{eqn:trivial_int_flow} gives $w\neq 0$ and $\sigma_qw=w$ for every $q\in (0,\infty)$. Thus, the commutator \hbox{$[\gamma,\widehat{\gamma}]\in \Sigma$} of any two elements $\gamma = (q,r,u)$ and $\widehat{\gamma} = (\widehat{q},\widehat{r},\widehat{u})$ in $\Gamma$ equals $(1,\ell b,\ell w)$, for some $\ell\in \Z$. By (\ref{eqn:operations_LH}-iii), \hbox{$(\sigma_q-1)\widehat{u} - (\sigma_{\widehat{q}}-1)u = \ell w \in \mathcal{E}_0 \cap\mathcal{E}_+ = \{0\}$}, cf. Lemma \ref{lem:E0andE+}(a). Consequently, $\ell=0$ and $[\gamma,\widehat{\gamma}] = (1,0,0)$ as required.

  Now comes the contradiction: if $\Gamma$ were Abelian, Lemmas \ref{lem:TCP_equiv}${\rm (b)}$ and \ref{teo:TCP_iso} would give $\Gamma \subseteq {\rm H}$ or $\Gamma \subseteq K$ for some subgroup $K$ of ${\rm G}_0$ isomorphic to $\R^k$, $k\in \{1,2\}$, acting freely on $\widehat{M}$. The former case leads to \eqref{eqn:t_descends}, while the latter contradicts Lemma \ref{lem:free_principal} as $\dim \widehat{M} > 2$.

\section{The four-dimensional case: proof of Corollary \ref{cor:dim4_goodbye}}\label{sec:4dim_goodbye}

Assume, on the contrary, that there exists a four-dimensional compact rank-one ECS manifold $(M,\g)$, and let $\Gamma$ be its fundamental group. As $(M,\g)$ is generic -- see the very end of Section \ref{sec:comp_Dperp} -- it must be translational by Theorem \ref{teoA:gen_cpt_translational}, while Corollary \ref{teoA:uc_cpct_model} allows us to identify its universal covering with a model $(\widehat{M},\hg)$ as in \eqref{eqn:ECS_data}--\eqref{eq:ECS_model}, so that $\Gamma\subseteq {\rm Iso}(\widehat{M},\hg)$ and $\widehat{M}/\Gamma = M$. Applying \eqref{eqn:passing_to} if necessary, we use Remark \ref{rem:conclusions_E} to conclude that $I=\R$ in \eqref{eqn:ECS_data} and that every $\gamma\in \Gamma$ has the form $\gamma=(1,p,{\rm Id},r,u)$, with $p\in \R$ and \hbox{$(r,u)\in {\rm H}$} (cf. Theorem \ref{teo:iso_heis}). Here, the ${\rm O}(V)$-component of $\gamma$ is assumed to be trivial due to genericity combined with \eqref{eqn:passing_to}. The image $\mathrm{P}$ of the homomorphism $\Gamma \ni \gamma\mapsto p\in \R$ is infinite cyclic as its being dense (or, trivial) would imply constancy of $f$ via the last condition in \eqref{eqn:group_Q} (or, lead to \eqref{eqn:t_descends}).  While ${\rm G} = \{  (1,p,{\rm Id},r,u)\in {\rm Iso}(\widehat{M},\hg) :  p\in \mathrm{P}\mbox{ and }(r,u)\in {\rm H}\}$ contains $\Gamma$ as a subgroup, no subgroup of ${\rm G}$ can act on $\widehat{M}$ freely and properly discontinuously with a compact quotient, as shown in \cite[Theorem 8.1]{AGAG10}.

\appendix

\section{How \cite[Lemma 7.3]{JGP08} leads to Theorem \ref{teoA:generic_mc}}\label{app:technical} 

With $u$ in \cite[Lemma 7.3]{JGP08} being ${\bf w}$ in \eqref{eqn:npg_D}, we fix \hbox{$y\colon I\to\widetilde{M}$} parametrized by $t$ (see \cite[Remark 7.2]{JGP08}) and consider the differential equation
\begin{equation}\label{eqn:diff_eq_Q}
 \nabla_{\!\dot{y}}\nabla_{\!\dot{y}}{\bf z} + R(\dot{y},{\bf z})\dot{y} + \nabla_{\!\dot{y}}\dot{y} = -\frac{Q({\bf z}){\bf w}}{4}, 
\end{equation}
where $Q({\bf z}) = 3\langle A({\bf z}+\widetilde{\mathcal{D}}), {\bf z}+\widetilde{\mathcal{D}}\rangle^{\boldsymbol \cdot} + 3f\tg({\bf z},{\bf z})^{\boldsymbol\cdot} + 2\,\,\dispdot{\!\!f} \tg({\bf z},{\bf z})$, imposed on sections ${\bf z}$ of $\widetilde{\mathcal{D}}^\perp$ along $y$. Here $\langle\cdot,\cdot\rangle$ stands for the inner product on the vector space $V$ of parallel sections of $\widetilde{\mathcal{D}}^\perp\!\!/\widetilde{\mathcal{D}}$. For a maximal solution ${\bf z}$ of \eqref{eqn:diff_eq_Q}, we set \hbox{$x(t,s) = \exp_{y(t)} s{\bf z}(t)$}, and observe that ${\bf z}$ and $x$ are now defined on $I$ and $I\times \R$ instead of $\R$ and $\R^2$ as in \cite{JGP08}: namely, a solution ${\bf z}_0$ of $\nabla_{\!\dot{y}}\nabla_{\!\dot{y}}{\bf z} + R(\dot{y},{\bf z})\dot{y} + \nabla_{\!\dot{y}}\dot{y}=0$ is clearly defined on all of $I$, while for a function $\mu\colon I\to \R$ with $\ddot{\mu} = Q({\bf z}_0)/4$, ${\bf z} = {\bf z}_0 - \mu{\bf w}$ is a solution to \eqref{eqn:diff_eq_Q}. With the subscript convention referred to in \eqref{eqn:derivatives_symmetric}, the vector field ${\bf v}$ along $x$ having ${\bf v}_s = 0$ for all $(t,s)$ and \hbox{${\bf v} = \nabla_{\!\dot{y}}\dot{y}$} for $s=0$ satisfies the conditions \hbox{$x_{tt} + (s-1)({\bf v} - Q(x_s){\bf w}/4) = 0$} and \hbox{$[Q(x_s)]_s = 0$}: they are precisely \cite[formula (9)]{JGP08}, being established here by the argument given there repeated \emph{verbatim}. It follows that \hbox{$I\ni t\mapsto x(t,1) \in \widetilde{M}$} is a maximal geodesic, which can be chosen to realize any $t$-normalized initial velocity.

\section{From \cite[Theorem 7.1]{JGP08} to Theorem \ref{teoA:uc_model}}\label{app:model}

Due to maximal completeness and \eqref{eqn:t-levels}, we may fix a maximal null geodesic $x\colon I\to \widetilde{M}$ parametrized by $t$ (cf. \cite[Remark 7.2]{JGP08}), and define a parallel field $\Pi$ of Lorentzian planes along $x$ by $\Pi_t = \R\dot{x}(t)\oplus\widetilde{\mathcal{D}}_{x(t)}$. As $\widetilde{\mathcal{D}}_{x(t)}^\perp = \Pi_t^\perp \oplus \widetilde{\mathcal{D}}_{x(t)}$ and $V$ is the space of parallel sections of the quotient bundle $\widetilde{\mathcal{D}}^\perp\hskip-2pt/\widetilde{\mathcal{D}}$, for every $t\in I$ we have an isomorphism $V \to \widetilde{\mathcal{D}}_{x(t)}^\perp/\widetilde{\mathcal{D}}_{x(t)} \to \Pi_t^\perp$, the image of $v\in V$ under which will be denoted by $v(t)$. Therefore, for each $t\in I$ and ${\bf w}$ as in \eqref{eqn:npg_D},
\begin{equation}\label{eqn:isoRV}
  \parbox{.83\textwidth}{$\R\times V \ni (s,v) \mapsto v(t) + \dfrac{s{\bf w}_{x(t)}}{2} \in \widetilde{\mathcal{D}}^\perp_{x(t)}$ is an obvious isomorphism.}
\end{equation}

By \cite[Lemma 5.1]{local_structure}, $F\colon \widehat{M} \to \widetilde{M}$ given by $F(t,s,v) = \exp_{x(t)}(v(t) + s{\bf w}_{x(t)}/2)$ has $F^*\tg = \hg$, and so, due to the italicized statement following \eqref{eqn:t-levels}, $F_\ast \widehat{\mathcal{D}} = \widetilde{\mathcal{D}}$. As the leaves of $\widetilde{\mathcal{D}}^\perp$ are simply connected -- see \cite[Theorem B]{DT_2} --  and their induced connections are flat (Remark \ref{obs:tg_leaves}), each slice $\{t\}\times \R\times V$ is diffeomorphically mapped under $F$ onto the leaf of $\widetilde{\mathcal{D}}^\perp$ passing through $x(t)$, so that injectivity of \eqref{eqn:isoRV} yields the one of $F$. Finally, the definition of maximal completeness and surjectivity of \eqref{eqn:isoRV} imply that $F$ is surjective as well. Hence, $F$ is a global isometry.

\bibliography{ECS_dilation_refs}{}
\bibliographystyle{plain}

\end{document}